\documentclass[11pt,reqno]{amsart}
\usepackage{amsmath,amsfonts,amssymb,amsthm}
\DeclareMathOperator{\csch}{csch}
\usepackage{graphicx}
\usepackage{xcolor}
\graphicspath{ }
\usepackage{wrapfig}
\usepackage[ocgcolorlinks, citecolor=cyan, linkcolor=magenta, urlcolor=violet]{hyperref}
\usepackage{accents}
\usepackage{caption}
\usepackage{subcaption}
\usepackage{calligra}

\numberwithin{figure}{section}

\theoremstyle{plain}
\newtheorem{thm}{Theorem}[section]
\newtheorem{lem}[thm]{Lemma}

\newtheorem{cor}{Corollary}[thm]
\theoremstyle{definition}

\newtheorem{exmp}{Example}[section]

\theoremstyle{remark}

\usepackage{mathtools}
\title{On triviality and scalar curvature estimation of gradient $h$-almost Yamabe solitons}
\author[A. A. Shaikh and A. Tripathi]{Absos Ali Shaikh$^{*1}$, Ananya Tripathi$^{2}$}
\address{$^{1,2}$Department of Mathematics,\newline University of
Burdwan, Golapbag,\newline Burdwan-713104,\newline West Bengal, India.}
\email{$^1$aashaikh@math.buruniv.ac.in, aask2003@yahoo.co.in}
\email{$^2$ananyatripathi999@gmail.com, atripathi@scholar.buruniv.ac.in}

\begin{document}

\begin{abstract}
	In this study we have explored gradient $h$-almost Yamabe solitons on both compact and complete non-compact Riemannian manifolds. We have established several sufficient conditions for the triviality of such solitons with respect to integral inequalities involving the scalar curvature and the soliton function. In this regard, we have acquired scalar curvature estimation under certain $L^1$ and $L^2$ integrability conditions, and defined signal of the function $h$. Our results extend and refine former works on almost and $h$-almost Yamabe solitons, and characterize the geometric structures of generalized Yamabe solitons.  
\end{abstract}

\noindent\footnotetext{ $^*$ Corresponding author.\\
$\mathbf{2020}$\hspace{5pt}Mathematics\; Subject\; Classification: 53C21; 53C24; 53C65. }
\maketitle

\section{Introduction}
The concept of Yamabe flow was originated by Hamilton \cite{RS1988} for resolving the Yamabe conjecture. The evolution equation of Yamabe flow on a Riemannian manifold is as follows:
\begin{equation}\label{Yamabe flow}
   \begin{cases}
	 \frac{\partial g(t)}{\partial t}=-R_{g(t)}g(t),\\
	  g(0) = g_0,
   \end{cases}
\end{equation} 
where $g(t)$ is a one-parameter family of metrics on the manifold and $R_{g(t)}$ denotes the scalar curvature corresponding to $g(t)$.

A natural generalization of self-similar solutions to the Yamabe flow leads to the notion of Yamabe solitons. A Riemannian manifold $(M^n, g)$, $n \geq 2$, is termed as a Yamabe soliton if it possesses a smooth vector field $X$ satisfying
\begin{equation}\label{Yamabe soliton}
   \frac{1}{2}\pounds_Xg=Rg-\lambda g,
\end{equation}
for some constant $\lambda \in \mathbb{R}$. Here, $\pounds_X$ is the Lie derivative in the direction of the vector field $X$ and $R$ is the scalar curvature corresponding to $g$. 
The Yamabe soliton is called a gradient Yamabe soliton if $X=\nabla f$, where $f:M^n\to\mathbb{R}$ is some smooth scalar field, and in this case the equation \eqref{Yamabe soliton} can be rewritten as
\begin{equation}\label{Gradient Yamabe soliton}
\nabla^2 f=Rg-\lambda g,
\end{equation}
where $\nabla^2 f$ is the Hessian of the potential function $f$.

Barbosa and Ribeiro \cite{BR13} initiated the study of almost Yamabe soliton by replacing the soliton constant $\lambda$ to a smooth scalar field on $M^n$. Later in 2021, Zeng \cite{ZE21} extended this conception to $h$-almost Yamabe soliton, providing a broader framework by incorporating a smooth scalar field $h$. An $h$-almost  Yamabe soliton, denoted by $(M^n,g,X,h,\lambda)$, is a Riemannian manifold $(M^n, g)$, equipped with a smooth vector field $X$, and satisfying the equation
\begin{equation}\label{$h$-almost Yamabe soliton}
	\frac{h}{2}\pounds_Xg=Rg-\lambda g,
\end{equation}
for some smooth scalar fields $h$ and $\lambda$. The function $h: M^n \to \mathbb{R}$ is said to have defined signal if, either $h>0$, or $h<0$. If $\lambda$ is constant then the soliton is termed as $h$-Yamabe soliton. The signal function (also referred to as weight function) $h$ acts geometrically as a local deformation weight that modulates how strongly the metric $g$ deforms along the potential vector field $X$.\\
When $\pounds_Xg = \pounds_{\nabla f}g$, for a smooth scalar field $f$ on $M^n$, the fundamental equation \eqref{$h$-almost Yamabe soliton} takes the form
\begin{equation}\label{E1}
	h\nabla^2 f=Rg-\lambda g,
\end{equation}
and in this instance, the soliton is termed as the $h$-almost gradient Yamabe soliton with potential function $f$, and is denoted by $(M^n,g,f,h,\lambda)$. If the potential field $X$ is Killing or the potential function $f$ is constant, then the soliton is trivial, otherwise non-trivial. It should be mentioned that an $h$-almost Yamabe soliton is expanding, shrinking, or steady according as $\lambda<0$, $\lambda>0$, or $\lambda=0$ on $M^n$, respectively, and that is undefined if $\lambda$ has no specific sign.
\\\noindent In terms of local coordinates, equation \eqref{E1} can be expressed as
\begin{equation}\label{E2}
h\nabla_i \nabla_j f = Rg_{ij}- \lambda g_{ij}.
\end{equation} 
An $h$-almost Yamabe soliton is a generalized geometric structure that models spatially varying, non-uniform scaling and deformation in curved spacetimes, and has a wide range of applications in various fields of mathematical physics such as cosmological models (Sultana–Dyer black-hole spacetimes \cite{AK26}, Generalized Robertson-Walker spacetimes), relativistic fluid spacetimes, static spacetimes and quantum wave equations.
We have derived some examples of $h$-almost gradient Yamabe soliton, including the case of expanding and shrinking.
\begin{exmp}
	Let us consider the warped product manifold 
	$M^n = I \times {}_{\cosh t} \mathbb{H}^{n-1}$ with the metric $g = dt^2 + \cosh^2t g_0$, where $I= (0,\infty)$, and $(\mathbb{H}^{n-1},g_0)$ is the $(n-1)$-dimensional hyperbolic space. Then, we can show that $Ric_g = -(n-1)g$, and hence the scalar curvature of $(M^n,g)$ is given by, $R_g= -n(n-1)$.\\
	Now, let us define a smooth scalar field $u: M^n \to \mathbb{R}$ by $u(t,x) = \sinh t$. Then, we can compute
	  \begin{equation*}
		 \nabla^2 u = \sinh tg, \quad 
		 \text{i.e.,} \quad	\csch t\nabla^2 u = g
	  \end{equation*}
	Therefore, $\csch t\nabla^2 u = (R_g - \lambda))g$, where $\lambda = -(n^2-n+1) < 0$. So, $(M^n,g,\nabla u, \csch t, \lambda)$ is an expanding $h$-almost gradient Yamabe soliton.\\
	Similarly, $((0,\frac{\pi}{2})\times {}_{\sin t} \mathbb{S}^{n-1},dt^2 + \sin^2t g_0,\nabla \cos t,\frac{1}{\cos t}, (n^2-n+1))$ is an example of shrinking $h$-almost gradient Yamabe soliton. To be precise, here
	\begin{eqnarray*}
		&&Ric_g = (n-1)g, \quad \text{i.e.,} \quad R_g= n(n-1),\\
	\text{and}
	    &&\nabla^2 \cos t = - \cos tg, \quad 
	    \text{i.e.,} \quad \frac{1}{\cos t}\nabla^2 \cos t = -g. 
	\end{eqnarray*}
\end{exmp}
Observe that both the aforementioned warped product spaces are incomplete. We now present examples of complete $h$-almost gradient Yamabe soliton.
\begin{exmp}
	Consider $M^n = \mathbb{R} \times {}_{\cosh t} \mathbb{R}^{n-1}$ with the warped product metric $g = dt^2 + \cosh^2 t g_{euc}$,  $g_{euc}$ being the flat Euclidean metric on $\mathbb{R}^{n-1}$. Then by a straight forward calculation, we can show that 
	$$Ric_g = -(n-1)g + (n-2)g_{euc}.$$
	Therefore, scalar curvature 
	$R_g= -(n-1)Tr_g(g) + (n-2)Tr_g(g_{euc}) = -n(n-1) + \frac{(n-1)(n-2)}{\cosh^2 t}$.\\
	Define $u: M^n \to \mathbb{R}$ by $u(t,x) = \sinh t +1$.
	Then, the Hessian of $u$ is given by
	\begin{equation*}
		\nabla^2 u = \sinh t g, \quad 
		\text{i.e.,} \quad	2\cosh t \nabla^2 u = \sinh 2tg.
	\end{equation*}
	Now, choose the non-constant smooth scalar fields:
	$$h (t,x) = 2\cosh t, \quad \text{and} \quad
	\lambda (t,x) = R_g - \sinh 2t.$$
	Then, $(M^n,g,\nabla u, h, \lambda)$ is a complete non-compact $h$-almost gradient Yamabe soliton.
\end{exmp}

\begin{exmp}
    Let us consider $n$-dimensional Euclidean space $\mathbb{R}^n$ equipped with standard flat metric $g_{euc} = dx_i \otimes dx_i$. Let $\beta = (b_1,...,b_n)$ be a non-zero vector and define the function $u: \mathbb{R}^n \to \mathbb{R}$ by
	\begin{equation*}
		u(x) = u(x_1,...,x_n) = a x_i^2 + \langle \beta,x  \rangle + c,
	\end{equation*}
    for some non-zero constants $a$ and $c$. Then it is easy to check that $\nabla^2 u = 2ag_{euc}$.\\
    Hence, $(\mathbb{R}^n,g_{euc},\nabla u,\frac{1}{a},-2)$ is a complete non-compact $h$-almost gradient Yamabe soliton, where $h$ is a constant function.  
\end{exmp}

According to the classical result, formulated by Chow, Lu and Ni \cite{CLN06} (see also \cite{CD08} and \cite{SYH2012}, for alternative proofs of this result), any compact gradient Yamabe soliton has constant scalar curvature $R=\lambda$, and hence it is trivial. For the complete non-compact case the scalar curvature $R=\lambda$ holds under certain additional conditions (see, for instance \cite{CA22}, \cite{LV2012}, \cite{MC21}). Zeng \cite{ZE21} combined the quasi Yamabe soliton, initiated by Leandro and Pina \cite{LP16}, with $h$-almost Yamabe soliton for the gradient case. As a matter of fact, any almost $m$-quasi gradient Yamabe soliton with potential function $u$ is actually an $h$-almost gradient Yamabe soliton admitting potential function $e^{-\frac{u}{m}}$, and $h =-me^{\frac{u}{m}}$ is of negative signal. In recent times, Cunha and Siddiqi \cite{CS23} have derived certain results on scalar curvature of such gradient ($-me^{\frac{u}{m}}$)-almost Yamabe solitons. Filho and Fernandes \cite{FFF24} also have studied the conformal geometry of gradient $m$-quasi Yamabe solitons reformulating its fundamental equation to the equation of such conformal solitons.\\

In this framework we have established certain conditions in respect to which triviality of the soliton occurs in both compact and complete non-compact cases (without boundary), and have estimated the scalar curvature  under suitable integral and geometric constraints. Henceforth, we have  considered $h$ to be of defined signal. In the context of \cite{ZE21}, for the case of $h<0$, these outcomes are also applicable for almost quasi gradient Yamabe soliton.\\
Our first result provides some sufficient integral conditions for the potential function $f$ to be constant. In the first case we have modified the integral condition imposed by Zeng \cite[Theorem 1.2.]{ZE21}. In the remaining two cases, we have prescribed the integral condition $\int_M \langle \nabla f,\nabla \frac{1}{h}R\rangle d\mu \geq 0$ along with specific sign of the soliton function $\lambda$ on $M^n$.
\begin{thm}\label{T1}
	For any compact $h$-almost gradient Yamabe soliton possessing potential function $f$, if any one of the following integral conditions holds then the soliton is trivial:
	\begin{enumerate}
		\item[(i)]$\int_M Ric(\nabla f,\nabla f)d\mu \leq 0$;
		\item[(ii)] The soliton is expanding, $R \leq \lambda,$ and $\int_M \langle \nabla f,\nabla \frac{1}{h}R\rangle d\mu \geq 0$;
		\item[(iii)] The soliton is shrinking, $R \geq \lambda,$ and $\int_M \langle \nabla f,\nabla \frac{1}{h}R\rangle d\mu \geq 0$.
		\end{enumerate}
\end{thm}
Next, we have prescribed some sufficient integral relations involving scalar curvature and soliton function to obtain triviality.  
\begin{thm}\label{T2}
	Assume that $(M^n,g,f,h,\lambda)$ is a compact $h$-almost gradient Yamabe soliton of dimension $n(\geq 3 )$. Then the soliton becomes trivial if either of the following holds:
	\begin{enumerate}
		\item[(i)] $h>0,$ and $\int_M \frac{1}{h}R^2 d\mu \geq \int_M \frac{1}{h}\lambda^2 d\mu$;
		\item[(ii)] $h<0,$ and $\int_M \frac{1}{h}\lambda^2 d\mu \geq \int_M \frac{1}{h}R^2 d\mu$.
	\end{enumerate}
\end{thm}
Further, in the remaining three results we have observed the complete non-compact $h$-almost gradient Yamabe soliton under previous analogous conditions along with some extra $L^1$ and $L^2$-regularity criteria. The following result is motivated by \cite{HY25}.
\begin{thm}\label{T3}
Suppose $(M^n,g,f,h,\lambda)$ is a complete non-compact $h$-almost gradient Yamabe soliton such that 
$|\nabla f|, \frac{1}{h}R \in L^2(M), 0 \leq R \leq \lambda$, and $\langle \nabla f,\nabla \frac{1}{h}\lambda \rangle \leq 0$. Then either the soliton is trivial, or it is isometric to $B^{n-1} \times \mathbb{R}$, where $B^{n-1}$ is homeomorphic to a regular level set of $f$ which is totally geodesic.
\end{thm}
The notation used herein is 
$L^p(M)= \{u :M^n \to \mathbb{R}:\int_M |u|^p d\mu < +\infty \}$, for each $p \geq 1$.
\begin{thm}\label{T4}
	Let $(M^n,g,f,h,\lambda)$ be a non-trivial complete non-compact $h$-almost gradient Yamabe soliton such that 
	$\frac{1}{h}, |\nabla f| \in L^2(M)$. If any one of the following
	\begin{itemize}
		\item[(i)] $R \geq \lambda,$ and $ \langle \nabla f,\nabla h \rangle \leq 0$;
		\item[(ii)] $R \leq \lambda,$ and $ \langle \nabla f,\nabla h\rangle \geq 0$
	\end{itemize}
	holds, then the soliton is isometric to $B^{n-1} \times \mathbb{R}$, where $B^{n-1}$ is homeomorphic to a regular level set of $f$ which is totally geodesic.
\end{thm}
\begin{thm}\label{T5}
	Assume that $(M^n,g,f,h,\lambda)$ is a non-trivial complete non-compact $h$-almost gradient Yamabe soliton of dimension $n$ $(\geq 3 )$ with $R, |\nabla f|\in L^2(M), \text{and } |\nabla \Delta f| \in L^1(M)$. Then the soliton is either scalar flat, or it is isometric to $B^{n-1} \times \mathbb{R}$, where $B^{n-1}$ is homeomorphic to a regular level set of $f$, provided any one of the following conditions holds:
	\begin{enumerate}
		\item[(i)] $h>0, \text{ and } R^2 \geq R\lambda$;
		\item[(ii)] $h<0, \text{ and } R^2 \leq R\lambda$.
	\end{enumerate}
\end{thm}
Lastly, as a consequence of Theorem \ref{T5} we have acquired the following corollary regarding Euclidean isometry of the soliton.
\begin{cor}\label{C1}
	Let a non-trivial complete non-compact $n$$(\geq 3)$-dimensional $h$-almost gradient Yamabe soliton, satisfying \eqref{E1}, admits non-negative scalar curvature $R$ with the additional conditions $R, |\nabla f|\in L^2(M);|\nabla \Delta f| \in L^1(M)$, and $\lambda= ch$, for some non-zero constant $c$. Then $R=0$, and the manifold $(M^n,g)$ is isometric with $\mathbb{R}^n$, if either
	 	\begin{enumerate}
	 	\item[(i)] $h>0$, and the soliton is expanding, or
	 	\item[(ii)] $h<0$, and the soliton is shrinking with $R < \lambda$.
	    \end{enumerate}  
\end{cor}

\section{Preliminaries}
 In this section we have proved some general results on gradient $h$-almost Yamabe soliton that are required to prove our core results. These results also generalize the outcomes of Zeng \cite{ZE21}, concerning gradient ($-me^{\frac{u}{m}}$)-almost Yamabe soliton.
\\First recall that the potential function $f$ of an $h$-almost gradient Yamabe soliton is a concircular scalar field with characteristic function $\frac{1}{h}(R-\lambda)$. Thus, it follows from \eqref{E1} that
\begin{equation}\label{E3}
	\nabla_Z \nabla f = \frac{1}{h}(R-\lambda)Z, \text{ for any vector field } Z.
\end{equation}
Now, taking the trace of equation $(\ref{E1})$ we get
\begin{equation}\label{E4}
	\Delta f = \frac{n}{h}(R-\lambda).
\end{equation}
\begin{lem}\label{L1}
	Let $(M^n,g,f,h,\lambda)$ be an $h$-almost gradient Yamabe soliton. Then the following relations hold:
\begin{eqnarray}
&& \frac{1}{2}\Delta|\nabla f|^2 = \frac{n}{h^2}(R-\lambda)^2 
   + \langle \nabla f,\nabla \frac{1}{h}(R-\lambda) \rangle;\label{E7}\\
&& Ric(\nabla f) = -(n-1)\nabla \frac{1}{h}(R-\lambda);\label{E8}\\
&& \frac{1}{2}\langle \nabla f,\nabla R \rangle +
   (n-1)\Delta \frac{1}{h}(R-\lambda) + \frac{1}{h}R(R-\lambda) =0.\label{E10}
\end{eqnarray}
\end{lem}
\begin{proof}
Using \eqref{E3}, we have
\begin{equation*}
	\langle \nabla|\nabla f|^2, Z \rangle
	= Z|\nabla f|^2 = Z \langle \nabla f,\nabla f \rangle 
	= 2\langle \nabla_Z f,\nabla f \rangle
	= 2 \langle\frac{1}{h}(R-\lambda)\nabla f, Z \rangle,
\end{equation*}
for any vector field Z, and this entails
\begin{equation}\label{E6}
	\frac{1}{2}\nabla|\nabla f|^2 = \frac{1}{h}(R-\lambda)\nabla f.
\end{equation}
Now, for any vector field $U$ and smooth scalar field $\phi$ on $M$,
\begin{equation}\label{E11}
	div ( \phi U) = \phi div (U) + \langle \nabla \phi, U \rangle.
\end{equation}
Therefore, taking the divergence on \eqref{E6}, and then applying \eqref{E4}, we obtain \eqref{E7}.
\\\noindent Again, since $\nabla g=0$, taking the covariant derivative with respect to an arbitrary vector field $X$ on both sides of \eqref{E1}, we get 
\begin{eqnarray}
\nonumber&&X \langle \nabla_Y \nabla f, Z \rangle 
	- \langle \nabla_{\nabla_X Y} \nabla f, Z \rangle
	- \langle \nabla_Y \nabla f, \nabla_X Z \rangle
	= (\nabla_X \frac{1}{h}(R-\lambda)) \langle Y, Z \rangle,\\
\text{and so, }
        && \langle \nabla^2_{X, Y} \nabla f, Z \rangle
           = \langle \nabla_X \nabla_Y \nabla f, Z \rangle
           - \langle \nabla_{\nabla_X Y} \nabla f, Z \rangle 
           = \langle (\nabla_X \frac{1}{h}(R-\lambda))Y, Z \rangle,
\end{eqnarray}
for any vector field $Y, Z$. As a consequence we have,
\begin{equation}
	\nabla^2_{X, Y} \nabla f = (\nabla_X \frac{1}{h}(R-\lambda))Y.
\end{equation}
Now, we calculate with a local frame $\{E_i\}$ which is parallel at a point $p \in M$:
\begin{eqnarray}
\nonumber	&&\langle Ric(\nabla f), X \rangle 
	=\langle \mathcal{R} (E_i, X)\nabla f, E_i \rangle 
	= \langle (\nabla^2_{E_i, X} - \nabla^2_{X, E_i})\nabla f, E_i \rangle \\
\nonumber	&=& \langle (\nabla_{E_i} \frac{1}{h}(R-\lambda))X, E_i \rangle - \langle (\nabla_X \frac{1}{h}(R-\lambda))E_i, E_i \rangle \\
\nonumber	&=& \langle (\nabla_{E_i} \frac{1}{h}(R-\lambda))E_i, X \rangle - (\nabla_X \frac{1}{h}(R-\lambda)) \langle E_i, E_i \rangle \\
	&=& \langle \nabla \frac{1}{h}(R-\lambda), X \rangle - n \langle \nabla \frac{1}{h}(R-\lambda), X \rangle. 
\end{eqnarray}
Therefore, we obtain \eqref{E8}.\\
Lastly, by second contracted Bianchi identity, $\frac{1}{2}\langle \nabla f,\nabla R \rangle = \frac{1}{2}\nabla f(R) = div Ric(\nabla f)$.
\\Also, we know that $div Ric(\nabla f) = div(Ric(\nabla f)) - \langle \nabla^2 f, Ric \rangle.$
\\\noindent Using equation \eqref{E2} we can write in tensorial notation that 
\begin{equation}
	\langle \nabla^2 f, Ric \rangle = g^{ik}g^{jl}\frac{1}{h}(R-\lambda) g_{ij} Ric_{kl} = \frac{1}{h}(R-\lambda) g^{kl}Ric_{kl}= \frac{1}{h}R(R-\lambda).
\end{equation}
From the above relations and \eqref{E8}, we acquire \eqref{E10}.
\end{proof}

\section{Proof of the results}

\subsection{Proof of the Theorem \ref{T1}.}
\begin{proof}
First, we integrate the equation \eqref{E7} over $(M^n,g)$ so that
\begin{equation}\label{E12}
	n\int_M \frac{1}{h^2}(R-\lambda)^2 d\mu= -\int_M \langle \nabla f,\nabla \frac{1}{h}(R-\lambda) \rangle d\mu,
\end{equation}
where $d\mu$ is the volume element of $(M^n,g)$.\\
Also, from \eqref{E8} we can infer that
\begin{equation}\label{E13}
	\int_M Ric(\nabla f,\nabla f)d\mu = -(n-1)\int_M\langle\nabla \frac{1}{h}(R-\lambda), \nabla f \rangle d\mu.
\end{equation}
Therefore, using the preceding two equations we have
\begin{equation}
	\int_M Ric(\nabla f,\nabla f)d\mu =	n(n-1)\int_M \frac{1}{h^2}(R-\lambda)^2 d\mu \geq 0.
\end{equation}
In view of our assumption on Ricci curvature, we get $R = \lambda$, and so applying \eqref{E4} we get $\Delta f=0$. Since, $(M^n,g)$ is compact, we can conclude that the soliton becomes trivial.\\
For the last two assertions, we will first apply the formula \eqref{E11}, and equation \eqref{E4} to obtain
\begin{equation}
	div( \frac{1}{h}\lambda \nabla f) = \frac{n}{h^2}\lambda(R-\lambda) + \langle \nabla f,\nabla \frac{1}{h}\lambda \rangle.
\end{equation}
Integrating the above equation over $M^n$, we get
\begin{equation}\label{E22}
	\int_M \langle \nabla f,\nabla \frac{1}{h}\lambda \rangle d\mu = -n \int_M \frac{1}{h^2}\lambda(R-\lambda) d\mu.
\end{equation}
Furthermore, using \eqref{E12} and \eqref{E22} we can write
\begin{equation}
	n\int_M \frac{1}{h^2}(R-\lambda)^2 d\mu =
	 -\int_M \langle \nabla f,\nabla \frac{1}{h}R \rangle d\mu 
	 - n \int_M \frac{1}{h^2}\lambda(R-\lambda) d\mu.
\end{equation}
\begin{equation}\label{E23}
\text{i.e.,} \quad
	n\int_M \frac{1}{h^2}R(R-\lambda) d\mu = 
	-\int_M \langle \nabla f,\nabla \frac{1}{h}R \rangle d\mu \leq 0,
\end{equation}
according to our hypothesis.\\
Also, in both of the cases $(ii)$ and $(iii)$, $R(R-\lambda) \geq 0$, and $R \neq 0$. Therefore, we must have $R=\lambda$.
\end{proof}

\subsection{Proof of Theorem \ref{T2}.}
\begin{proof}
	Integrating by parts over $M^n$, we get
\begin{equation}
	\int_M R \Delta f d\mu = -\int_M\langle \nabla f,\nabla R \rangle d\mu.
\end{equation}
Now, multiplying both sides of \eqref{E4} by $R$ we integrate it over $M^n$, and substitute it in the above equation. Therefore,
\begin{equation}
	\int_M \frac{1}{h}R(R-\lambda) d\mu = -\frac{1}{n}\int_M\langle \nabla f,\nabla R \rangle d\mu.
\end{equation}
On the other side, integrating \eqref{E10} we have
\begin{equation}
	\int_M \frac{1}{h}R(R-\lambda) d\mu = -\frac{1}{2}\int_M\langle \nabla f,\nabla R \rangle d\mu.
\end{equation}
Since $n > 2$, we infer that
\begin{equation}\label{E14}
	\int_M\langle \nabla f,\nabla R \rangle d\mu = \int_M \frac{1}{h}R(R-\lambda) d\mu = 0.
\end{equation}
Now, when $h>0$, applying \eqref{E14} we get the following: 
\begin{eqnarray}
	0 \leq \int_M \frac{1}{h}(R-\lambda)^2 d\mu &=& 2\int_M \frac{1}{h}R(R-\lambda) d\mu - \int_M \frac{1}{h}R^2 d\mu + \int_M \frac{1}{h}\lambda^2 d\mu \\
	&=& \int_M \frac{1}{h}\lambda^2 d\mu - \int_M \frac{1}{h}R^2 d\mu.
\end{eqnarray}
Therefore, $R= \lambda$ holds, provided the inequality of case $(i)$ occurs. Hence, the soliton is trivial.
\\ \noindent In a similar way, we can prove this for the case $(ii)$.
\end{proof}

\subsection{Proof of Theorem \ref{T3}.}
\begin{proof}
	Consider the cut-off function $\eta_r \in C^\infty_0 (\mathbb{B}_{r}(p))$ as in \cite{CC96} satisfying,
\begin{equation}\label{cond1}
	\begin{cases} 
		0\leq \eta_r\leq 1 &\text{ everywhere} \\
		\eta_r = 1 & \text{ in }\mathbb{B}_{\frac{r}{2}}(p) \\
		supp(\eta_r) \subset \mathbb{B}_{r}(p) \\
		|\nabla \eta_r| \leq \frac{d}{r} & \text{ in }\mathbb{B}_r (p)
	\end{cases}
\end{equation}
where $\mathbb{B}_{r}(p)$ denotes the geodesic open ball of radius $r>0$  centered at $p \in M^n$, and $d$ is a positive constant.
\\\noindent Since, $\int_M div(\eta^2_r \frac{1}{h}(R-\lambda)\nabla f) = 0$, using \eqref{E4} we obtain
\begin{equation}\label{E15}
	n\int_M \eta^2_r \frac{1}{h^2}(R-\lambda)^2 d\mu
   + \int_M \eta^2_r \langle\nabla \frac{1}{h}(R-\lambda), \nabla f \rangle d\mu
   + 2\int_M \eta_r \frac{1}{h}(R-\lambda)\langle \nabla \eta_r,\nabla f \rangle d\mu 
   = 0.
\end{equation}
This implies
\begin{equation}\label{E16}
	0 \leq n\int_M \eta^2_r \frac{1}{h^2}(R-\lambda)^2 d\mu
	\leq - \int_M \eta^2_r \langle\nabla \frac{1}{h}R, \nabla f \rangle d\mu
	- 2\int_M \eta_r \frac{1}{h}(R-\lambda)\langle \nabla \eta_r,\nabla f \rangle d\mu, 
\end{equation}
because of our hypothesis.\\
Furthermore, in the same way as we get \eqref{E15}, we can have
\begin{equation}
	-\int_M \eta^2_r \langle\nabla \frac{1}{h}R, \nabla f \rangle d\mu
	= n\int_M \eta^2_r \frac{1}{h^2}R(R-\lambda) d\mu
	+ 2\int_M \eta_r \frac{1}{h}R\langle \nabla \eta_r,\nabla f \rangle d\mu. 
\end{equation}
Therefore, applying the Cauchy-Schwarz's inequality, and using our assumptions $\frac{1}{h}R, |\nabla f| \in L^2(M), 0 \leq R \leq \lambda$, we acquire
\begin{eqnarray}\label{E17}
	\nonumber - \int_M \eta^2_r \langle\nabla \frac{1}{h}R, \nabla f \rangle d\mu
	&\leq&  2\int_M \eta_r \frac{1}{h}R\langle \nabla \eta_r,\nabla f \rangle d\mu\\
	&\leq& \nonumber 2 {\biggl(\int_M \eta_r^2 \frac{1}{h^2}R^2 d\mu \biggr)}^\frac{1}{2}
	          {\biggl(\int_M|\nabla\eta_r|^2|\nabla f|^2 d\mu\biggr)}^\frac{1}{2}\\
	&\leq&  \frac{d_1}{r},  \quad \text{for some constant $d_1>0$}.      
\end{eqnarray}
Again, applying the Cauchy-Schwarz's and  Young's inequalities, we obtain
\begin{eqnarray}\label{E18}
\nonumber &&- 2\int_M \eta_r \frac{1}{h}(R-\lambda)\langle \nabla \eta_r,\nabla f\rangle d\mu\\
	&\leq& \nonumber 2{\biggl(\int_M \eta_r^2 \frac{1}{h^2}(R-\lambda)^2d\mu\biggr)}^\frac{1}{2}
    {\biggl(\int_M|\nabla\eta_r|^2|\nabla f|^2 d\mu\biggr)}^\frac{1}{2}\\
    &\leq& \nonumber 
    \int_M \eta_r^2 \frac{1}{h^2}(R-\lambda)^2 d\mu + \int_M|\nabla\eta_r|^2|\nabla f|^2 d\mu\\
    &\leq&  
    \int_M \eta_r^2 \frac{1}{h^2}(R-\lambda)^2 d\mu + \frac{d_2}{r^2},  \quad \text{for some constant $d_2>0$}.
\end{eqnarray}
Furthermore, $\eta_r = 1$ in $\mathbb{B}_{\frac{r}{2}}(p)$. Hence, using the equations \eqref{E16}, \eqref{E17} and \eqref{E18} we can conclude that, as $r \to \infty$, $R = \lambda$ holds.
Therefore, by \eqref{E1} either the soliton is trivial, or $f$ is a non-trivial affine function. In the second case, $M^n$ is isometric to $B^{n-1} \times \mathbb{R}$ (see, \cite{IN82}), where $B^{n-1}$ is homeomorphic to a regular level set of $f$ which is totally geodesic.
\end{proof}

\subsection{Proof of Theorem \ref{T4}.}
\begin{proof}
As in the Theorem \ref{T3} we consider the same cut-off function  $\eta_r \in C^2_0 (\mathbb{B}_{r}(p))$, satisfying the conditions in \eqref{cond1}. Then, we have $|\nabla \eta^2_r| \leq \frac{2d}{r} \text{ in }\mathbb{B}_r (p)$.\\
Now, from \eqref{E4} it follows that
\begin{equation}\label{E24}
	n\int_M \eta^2_r \frac{1}{h^2}(R-\lambda) d\mu
	= \int_M \eta^2_r \frac{1}{h}\Delta f d\mu
	= \int_M\eta^2_r \frac{1}{h^2}\langle \nabla f,\nabla h \rangle d\mu
	  - \int_M  \frac{1}{h}\langle \nabla f, \nabla \eta_r^2 \rangle d\mu.
\end{equation}
If $(i)$ holds, then using the $L^2$-integrability conditions, and equation \eqref{E24} we have
\begin{eqnarray}
	\nonumber
	0 \leq n\int_M \eta^2_r \frac{1}{h^2}(R-\lambda) d\mu
	\leq - \int_M  \frac{1}{h}\langle \nabla f, \nabla \eta_r^2 \rangle d\mu
	&\leq& {\biggl(\int_M  \frac{1}{h^2} d\mu\biggr)}^\frac{1}{2}
	{\biggl(\int_M|\nabla\eta_r^2|^2|\nabla f|^2 d\mu\biggr)}^\frac{1}{2}\\
	&\leq& \frac{2d}{r}{\biggl(\int_M  \frac{1}{h^2} d\mu\biggr)}^\frac{1}{2}
	{\biggl(\int_M |\nabla f|^2 d\mu\biggr)}^\frac{1}{2} \to 0,
\end{eqnarray}
as $r \to \infty$. Since, $\eta_r = 1$ in $\mathbb{B}_{\frac{r}{2}}(p)$ so we must have $R=\lambda$.\\
By a similar argument that we have used for case $(i)$, we can show that $R=\lambda$ holds for case $(ii)$.\\
Therefore, $f$ is a non-constant affine function, and so by \cite{IN82} it is isometric to $B^{n-1} \times \mathbb{R}$, where $B^{n-1}$ is homeomorphic to a regular level set of $f$.
\end{proof}

\subsection{Proof of Theorem \ref{T5}.}
\begin{proof}
We use the same cut-off function $\eta_r \in C^2_0 (\mathbb{B}_{r}(p))$ that we have considered in Theorem \ref{T3}, fulfilling the conditions in \eqref{cond1}. Then, $|\nabla \eta^2_r| \leq \frac{2d}{r} \text{ in }\mathbb{B}_r (p)$.
\\\noindent Now, multiplying both sides of \eqref{E10} by $\eta^2_r$, and then integrating it over $M^n$, we can write
\begin{equation}\label{E19}
    \frac{1}{2} \int_M\eta^2_r\langle \nabla f,\nabla R \rangle d\mu  + (n-1)\int_M\eta^2_r\Delta \frac{1}{h}(R-\lambda) d\mu + \int_M\eta^2_r\frac{1}{h}R(R-\lambda) d\mu = 0.
\end{equation}
Again, integrating by parts, and using \eqref{E4} we get
\begin{equation}\label{E20}
    \int_M\eta^2_r\langle \nabla f,\nabla R \rangle d\mu + n\int_M \eta^2_r \frac{1}{h}R(R-\lambda) d\mu + 2\int_M \eta_r R\langle \nabla \eta_r,\nabla f \rangle d\mu = 0.  
\end{equation}
Combining the above two equations with \eqref{E4}, and using Cauchy-Schwarz's inequality, we obtain
\begin{eqnarray}\label{E21}
\nonumber
        (\frac{n}{2}-1)\int_M \eta^2_r \frac{1}{h}R(R-\lambda) d\mu
\nonumber &=&
        (n-1)\int_M \eta^2_r \Delta\frac{1}{h}(R-\lambda) d\mu - 
          \int_M \eta_r R\langle \nabla \eta_r,\nabla f \rangle d\mu \\
\nonumber &=&
        -(n-1)\int_M \langle \nabla \frac{1}{h}(R-\lambda), \nabla\eta^2_r\rangle d\mu - 
        \int_M \eta_r R\langle \nabla \eta_r,\nabla f \rangle d\mu \\
\nonumber &\leq&
        (n-1)\int_M|\nabla \frac{1}{h}(R-\lambda)|
        |\nabla \eta^2_r| d\mu - \int_M \eta_r R\langle \nabla \eta_r,\nabla f \rangle d\mu \\
\nonumber &\leq&
        \frac{(n-1)}{n}\int_M|\nabla \Delta f|
        |\nabla \eta^2_r| d\mu + 
          {\biggl(\int_M \eta_r^2 R^2 d\mu\biggr)}^\frac{1}{2}
          {\biggl(\int_M|\nabla\eta_r|^2|\nabla f|^2 d\mu\biggr)}^\frac{1}{2}\\
         &\leq&
        \frac{2d(n-1)}{nr}\int_M|\nabla \Delta f| d\mu 
        + \frac{d}{r}{\biggl(\int_M R^2 d\mu\biggr)}^\frac{1}{2}
        {\biggl(\int_M|\nabla f|^2 d\mu\biggr)}^\frac{1}{2}.  
\end{eqnarray}
Using our $L^1$-integrability assumption on the first term, and $L^2$-integrability assumptions on the second term in the right hand side (RHS) of \eqref{E21}, we can conclude that the RHS tends to $0$ as $r \to \infty$.\\
Moreover, both of the cases $(i)$ and $(ii)$ implies
\begin{equation*}
	 \frac{1}{h}R(R-\lambda) \geq 0,
\end{equation*}
and $\eta_r = 1$ in $\mathbb{B}_{\frac{r}{2}}(p)$, so we must have $R(R-\lambda) = 0$. Consequently, either $R=0$, or $R=\lambda$. Since, the soliton is non-trivial so $R=\lambda$ implies that $f$ is a non-constant affine function, and so by \cite{IN82} it is isometric to $B^{n-1} \times \mathbb{R}$, where $B^{n-1}$ is homeomorphic to a regular level set of $f$.
\end{proof}	

\subsection{Proof of Corollary \ref{C1}.}
\begin{proof}
	Both our prescribed conditions implies that $\frac{1}{h}(R - \lambda) >0 $.\\
	Again, since $R$ is non-negative so using the $L^2$-integrability assumption we can conclude from \eqref{E21} that, $R=0$.
	Therefore, \eqref{E1} takes the form
	\begin{eqnarray}
		\nabla^2 f= -\frac{1}{h}\lambda g, \quad
		\text{i.e.,} \quad \nabla^2 f = -cg.
	\end{eqnarray}
    Lastly, we apply Tashiro's rigidity theorem \cite[Theorem 2]{TY65} which yields that, a complete Riemannian manifold $(M,g)$ possessing a concircular scalar field $\phi$ such that $\nabla^2\phi= ag$, for some non-zero constant $a$, is isometric to a Euclidean space. Therefore,  we get the result.
\end{proof}

\section{acknowledgment}
The second author greatly acknowledges to The University Grants Commission, Government of India for the award of Junior Research Fellowship.

\end{document}